\documentclass[11pt, reqno]{amsart}
\usepackage{amsmath, amssymb}
\usepackage{hyperref}

\usepackage{color}
\usepackage{bm, mathrsfs, enumitem}
\usepackage[normalem]{ulem}
\usepackage{graphicx}
\usepackage{tikz}
\usepackage[all]{xy}
\allowdisplaybreaks

\newtheorem{theorem}{Theorem}[section]
\newtheorem{lemma}[theorem]{Lemma}
\newtheorem{proposition}[theorem]{Proposition}
\newtheorem{corollary}[theorem]{Corollary}

\theoremstyle{definition}

\newtheorem{example}[theorem]{Example}

\theoremstyle{remark}
\newtheorem{remark}[theorem]{Remark}
\numberwithin{equation}{section}

\newcommand{\Z}{\mathbb{Z}}

\newcommand{\R}{\mathbb{R}}

\newcommand{\ab}{\mathrm{ab}}
\newcommand{\cd}{\mathrm{cd}}
\newcommand{\defy}{\mathrm{def}}

\newcommand{\BS}{\mathit{BS}}
\newcommand{\FP}{\mathit{FP}}

\newcommand{\Int}{\operatorname{Int}}
\newcommand{\Cone}{\operatorname{Cone}}
\newcommand{\Hom}{\mathrm{Hom}}

\newcommand{\ang}[1]{\langle #1 \rangle}

\title[]{On the BNSR invariants of $2$-knot groups}

\author{Yuta Nozaki}
\address{
Department of Mathematics, Faculty of Science, Hokkaido University \\
Sapporo 060-0810 \\
Japan\vspace{-0.6em}}
\address{
International Institute for Sustainability with Knotted Chiral Meta Matter (WPI-SKCM$^2$), Hiroshima University \\
1-3-1 Kagamiyama, Higashi-Hiroshima, Hiroshima 739-8531 \\
Japan}
\email{nozaki@math.sci.hokudai.ac.jp}

 %20F65

\begin{document}

\begin{abstract}
For a finitely generated group $G$, the Bieri--Neumann--Strebel--Renz (BNSR) invariants are subsets of the character sphere of $G$ that govern the finiteness properties of normal subgroups containing the commutator subgroup. 
We investigate the BNSR invariants of $2$-knot groups and link groups. 
As an application, we prove Yasuda's conjecture that every non-trivial $2$-plat $2$-knot is non-fibered.
\end{abstract}

\maketitle
%\setcounter{tocdepth}{1}
%\tableofcontents

%%%%%%%%%
\section{Introduction}
\label{sec:Introduction}

The study of finiteness properties of groups has long been an important topic in group theory. 
A landmark development in this field was the introduction of the BNS invariant $\Sigma^1(G)$ for a finitely generated group $G$ by Bieri, Neumann, and Strebel~\cite{BNS87} (see \cite[Theorem~C1.13]{Str13}).
Let $S(G)$ denote the set $(\Hom(G, \R)\setminus \{0\})/\R_{>0}$ of non-trivial characters $G \to \R$ up to multiplication by positive scalars.
The BNS invariant is defined as a subset of the character sphere $S(G)$ and provides a geometric criterion to determine which kernels of characters $\chi\colon G \to \R$ are finitely generated. 
In particular, a normal subgroup $N$ containing the commutator subgroup $[G, G]$ is finitely generated if and only if the characters vanishing on $N$ lie in $\Sigma^1(G)$.
Building upon this framework, Bieri and Renz~\cite{BiRe88} generalized the original BNS invariant to higher dimensions (see also \cite{Ren88}).
These are categorized into two types: homological BNSR invariants $\Sigma^m(G; R)$ and homotopical BNSR invariants $\Sigma^m(G)$ for positive integers $m$, where $R$ is a non-zero commutative ring with unity.
The invariant $\Sigma^m(G; R)$ (resp.~$\Sigma^m(G)$) characterizes whether the kernel of a character satisfies the homological finiteness property $\FP_m$ over $R$ (resp.~homotopical finiteness property $F_m$).
We refer the reader to Geoghegan's book~\cite[Chapters~7 and~8]{Geo08} for finiteness properties and to a comprehensive survey \cite{Str13} by Strebel for details of the BNS invariant.

In the realm of low-dimensional topology, the BNS invariants have emerged as a fundamental bridge between the algebraic structure of the fundamental group and the geometric properties of the underlying manifold.
See, for instance, Kochloukova~\cite{Koc06JPAA}, Kielak~\cite{Kie20}, Friedl and Vidussi~\cite{FrVi23}.
As another remarkable application, Ershov and He~\cite{ErHe18} computed the BNS invariant of the Torelli subgroup of the mapping class group of a certain surface.
Their proof is based on Brown's characterization~\cite{Bro87} of $\Sigma^1$ in terms of actions on $\R$-trees, and they conclude that the Johnson kernel is finitely generated.
This result was subsequently generalized by Church, Ershov, and Putman~\cite{CEP22}.

In this paper, we study the BNSR invariants of $2$-knot groups and classical link groups.
Let $L$ be a $d$-link in the $(d+2)$-sphere $S^{d+2}$, that is, a smooth submanifold of $S^{d+2}$ which is diffeomorphic to a disjoint union of $d$-spheres.
We simply call it a link when $d=1$, and $L$ is called a $d$-knot if it is connected.
We write $E(L)$ for the exterior of $L$ and $G(L)$ for the fundamental group of $E(L)$.

It is shown in \cite[Corollary~F]{BNS87} that, for a $3$-manifold $M$, $\Sigma^1(\pi_1(M))$ is symmetric, i.e., invariant under the multiplication by $-1$.
In particular, if $H_1(M; \Z)\cong \Z$, we conclude that $\Sigma^1(\pi_1(M))$ is either empty or $\{[\ab], [-\ab]\}$, where $\ab\colon \pi_1(M)\to \Z$ is the abelianization.

\begin{proposition}[folklore]\label{prop:folklore}
For a knot $K$ in $S^3$ and a positive integer $m$, $\Sigma^m(G(K))$ is $S(G(K))$ if $K$ is fibered, empty if not.
The same holds for $\Sigma^m(G(K); R)$.
\end{proposition}

This proposition immediately follows from a known property of $\Sigma^m$ and Stallings' fibering theorem~\cite{Sta62}. 
Now, it is natural to investigate the BNSR invariants of $2$-knot groups.
While $E(K)$ is always an Eilenberg--MacLane space for a knot $K$ in $S^3$, it is not true for $2$-knots since $2$-knot groups can have torsion elements (see Example~\ref{ex:Zee65}).
Furthermore, the fundamental group of a $4$-manifold is not necessarily of type $F_\infty$ since every finitely presentable group can be realized as the fundamental group of a compact $4$-manifold.
In contrast, for the exteriors of fibered $2$-knots (more generally for $4$-manifolds fibering over $S^1$), we obtain the following.

\begin{proposition}
\label{prop:fibered}
Let $X$ be a topological $4$-manifold admitting a fiber bundle structure over $S^1$ such that a fiber is a compact connected $3$-manifold and $H_1(X; \Z) \cong \Z$. 
Then $G = \pi_1(X)$ and $[G, G]$ are of type $F_\infty$.
Consequently, $\Sigma^m(G) = \Sigma^m(G; R) = S(G)$ for $m\geq 1$ and a non-zero commutative ring $R$ with unity.
\end{proposition}

For instance, $k$-twist-spun knots constructed from knots in $S^3$ are fibered if $k\in \Z\setminus \{0\}$.
Yasuda~\cite{Yas26} studies in detail another interesting infinite family of $2$-knots, namely $2$-plat $2$-knots, which are constructed from $2$-dimensional braids.

\begin{theorem}
\label{thm:2-plat}
Let $K$ be a non-trivial $2$-plat $2$-knot.
Then $\Sigma^1(G(K)) \subsetneq S(G(K))$.
\end{theorem}

For instance, there exists a $2$-plat $2$-knot whose knot group is isomorphic to the Baumslag--Solitar group $\BS(1,2)$, and $\Sigma^1(\BS(1,2))$ is a singleton (see Example~\ref{ex:BS}).
In contrast to the BNS invariant of $3$-manifold groups, $\Sigma^1(G(K))$ can be non-symmetric for a $2$-knot $K$.
Combining Proposition~\ref{prop:fibered} and Theorem~\ref{thm:2-plat}, we prove Yasuda's conjecture~\cite[Conjecture~1.13]{Yas26}, which was motivated by computer calculations of Alexander polynomials.

\begin{corollary}
\label{cor:Yasuda}
Every non-trivial $2$-plat $2$-knot is non-fibered.
\end{corollary}

Finally, we briefly survey the BNSR invariants of $G(L)$ for links $L$ in $S^3$ and show the following.

\begin{proposition}
\label{prop:Hopf}
Let $L$ be a link in $S^3$ with at least two components and let $G = G(L)$.
Then $\Sigma^1(G) = S(G)$ if $L$ is the Hopf link and $\Sigma^1(G) \subsetneq S(G)$ otherwise.
In other words, $[G, G]$ is finitely generated if and only if $L$ is a Hopf link. 
\end{proposition}

The organization of this paper is as follows.
In Section~\ref{sec:Preliminaries}, we recall the definitions and fundamental properties of the BNSR invariants.
Section~\ref{sec:Proofs} is devoted to the introduction of $2$-plat $2$-knots and the proofs of our main results. 
In Section~\ref{subsec:difference}, we discuss further group-theoretic differences between $2$-knot groups and classical knot groups, and we investigate the BNSR invariants of link groups in $S^3$.

%%%%%%%%%%%%
\subsection*{Acknowledgments}
The author would like to thank Akihiro Takano and Jumpei Yasuda for fruitful discussions.
He also thanks Takahiro \mbox{Kitayama}, Mizuki Fukuda, Toshiyuki Akita, Takuya Sakasai, and Masatoshi Sato for helpful comments.
This study was supported in part by JSPS KAKENHI Grant Number JP23K12974.
%%%%%%%%%%%%

\section{Preliminaries}
\label{sec:Preliminaries}

\subsection{BNSR invariants}
In this subsection, we briefly review the definition of the BNSR invariants and their properties used in this paper.
We refer the reader to \cite{BNS87}, \cite{BiRe88}, \cite[Section~18.3]{Geo08}, \cite{KoVi25}, and \cite{Str13}.
Let $G$ be a finitely generated group.
For a character $\chi\colon G \to \R$, we write $G_\chi$ for the submonoid consisting of $g$ such that $\chi(g)\geq 0$.
Let $R$ be a non-zero commutative ring with unity and let $m$ be a positive integer.
Then, the \emph{homological BNSR invariant} $\Sigma^m(G;R)$ is defined by
\[
\Sigma^m(G; R) = \{ [\chi]\in S(G)\mid \text{$R$ is of type $\FP_m$ as a left $R[G_\chi]$-module} \},
\]
where $G$ acts on $R$ trivially.

For a CW complex $\Gamma$, let $\Gamma^{(m)}$ denote its $m$-skeleton.
When the vertices are labeled by elements of $G$, we write $\Gamma_\chi$ for the full subcomplex spanned by vertices $g$ with $\chi(g)\geq 0$.
Then, the \emph{homotopical BNSR invariant} $\Sigma^m(G)$ is defined by
\[
\Sigma^m(G) = \left\{ [\chi]\in S(G) \biggm| \parbox{22em}{there is a $K(G,1)$ space with a single $0$-cell 
and finite $m$-skeleton such that $\Gamma^{(m)}_\chi$ is $(m-1)$-connected} \right\},
\]
where $\Gamma$ is the universal cover of the $K(G,1)$ space.
As shown in \cite[IV.1, Korollar~1.6]{Ren88}, we have $\Sigma^m(G) \subset \Sigma^m(G; R)$ and equality holds when $m=1$.
Moreover, $\Sigma^m(G) = \Sigma^m(G; \Z)\cap \Sigma^2(G)$ for $m\geq 2$ is found in \cite[I.3, Satz~B]{Ren88}.
The following results are most fundamental.

\begin{theorem}[{\cite[Theorem~B]{BiRe88}}]
\label{thm:homological_BNS}
Let $G$ be a group of type $\FP_{m_0}$ over $R$ and $N$ a normal subgroup of $G$ containing $[G, G]$, and $1\leq m \leq m_0$.
Then, $N$ is of type $\FP_m$ over $R$ if and only if $\{[\chi] \in S(G) \mid \chi|_N = 0\} \subset \Sigma^m(G; R)$.
\end{theorem}

See also \cite[Theorem~5]{Koc06JPAA}.

\begin{theorem}[{\cite[I.3, Satz~C]{Ren88}}]
\label{thm:homotopical_BNS}
Let $G$ be a group of type $F_{m_0}$ and $N$ a normal subgroup of $G$ containing $[G, G]$, and $1\leq m \leq m_0$.
Then, $N$ is of type $F_m$ if and only if $\{[\chi] \in S(G) \mid \chi|_N = 0\} \subset \Sigma^m(G)$.
\end{theorem}

We will use \cite[Theorem~4.2]{Bro87} to compute the BNS invariant of a two-generator one-relator group.
See also \cite[Theorem~B4.1]{Str13}.
Let $G$ be a group given by a presentation $\ang{a, b \mid r}$ such that $r = s_1 \cdots s_k$ is a cyclically reduced, non-empty word involving both generators.
Here, the word $r$ is said to be cyclically reduced if $(s_{i}, s_{i+1})$ is neither $(a^{\pm 1}, a^{\mp 1})$ nor $(b^{\pm 1}, b^{\mp 1})$ for any $i =1, \dots, k$, where indices are considered modulo $k$.

\begin{theorem}[{\cite[Theorem~4.2]{Bro87}}]
\label{thm:Bro87}
Let $G$ be a group as above and let $\chi \colon G \to \R$ be a non-zero character.
Then, $[\chi] \in \Sigma^1(G)$ if and only if the sequence $\chi(s_1), \chi(s_1 s_2), \dots, \chi(s_1 \cdots s_k)$ satisfies the following condition:
if one of $\chi(a)$ and $\chi(b)$ is zero, the sequence attains its minimum exactly twice,
otherwise the sequence attains its minimum exactly once.
\end{theorem}

%%%%%%%%%
\section{Proofs of main results}
\label{sec:Proofs}

\subsection{$2$-knots in $S^4$}
\label{subsec:2-knot}

In this subsection, we prove Proposition~\ref{prop:fibered} and discuss a $k$-twist-spun knot as an example.

\begin{lemma}\label{lem:3-manifold}
Let $Y$ be a compact connected $3$-manifold.
Then $\pi_1(Y)$ is of type $F_\infty$.
\end{lemma}

\begin{proof}
First, we may assume that $Y$ is orientable.
Indeed, if $Y$ is non-orientable, we consider its orientation double cover $Y_2$.
Since $\pi_1(Y_2)$ is a subgroup of $\pi_1(Y)$ of index $2$, by \cite[Corollary~7.2.4]{Geo08}, $\pi_1(Y)$ is of type $F_\infty$ if and only if $\pi_1(Y_2)$ is of type $F_\infty$.

Since the free product of two groups of type $F_\infty$ is again of type $F_\infty$, it suffices to show the case where $Y$ is a prime $3$-manifold.
In such a case, $Y$ is either irreducible or $S^2\times S^1$ (see \cite[Lemma~3.13]{Hem04}).
For the latter case, $\pi_1(Y) \cong \Z$ is of type $F_\infty$.
Consider the former case.
If $G = \pi_1(Y)$ is finite, then it is of type $F_\infty$ by \cite[Corollary~7.2.5]{Geo08}.
Assume $G$ is infinite.
Since $Y$ is irreducible, the sphere theorem implies $\pi_2(Y)=0$.
It follows from $|G| = \infty$ that $H_3(\widetilde{Y}; \Z)=0$ for the universal cover $\widetilde{Y}$ of $Y$.
Hence, the Hurewicz theorem shows that $Y$ is a $K(G,1)$ space, and thus $G$ is of type $F_\infty$.
\end{proof}

\begin{proof}[Proof of Proposition~\ref{prop:fibered}]
Let $X$ be a topological $4$-manifold admitting a fiber bundle structure over $S^1$ such that a fiber is a compact connected $3$-manifold $\Sigma$ and $H_1(X; \Z) \cong \Z$. 
We then have an exact sequence $1 \to \pi_1(\Sigma) \to \pi_1(X) \to \pi_1(S^1) \to 1$.
Since the induced epimorphism $H_1(X; \Z) \to \Z$ is an isomorphism, $\pi_1(\Sigma)$ is identified with $[\pi_1(X), \pi_1(X)]$.
It follows from Lemma~\ref{lem:3-manifold} that $\pi_1(\Sigma)$ is of type $F_\infty$.
By \cite[Exercise~1 in Section~7.2]{Geo08}, $\pi_1(X)$ is also of type $F_\infty$.

Finally, the second assertion follows from Theorem~\ref{thm:homotopical_BNS}.
\end{proof}

\begin{example}
\label{ex:Zee65}
We first recall basic facts about twist-spun knots from \cite{Zee65}.
For a knot $K$ in $S^3$ and an integer $k$, one can construct a $2$-knot $K^k$ in $S^4$ called the \emph{$k$-twist-spun knot}.
It is known that, if $k\neq 0$, $K^k$ is a fibered knot whose fiber is the $|k|$-fold cyclic branched cover $\Sigma_{|k|}(K)$ of $S^3$ branched over $K$ with an open $3$-ball removed.
Let $K$ be the trefoil knot. 
Then, by Proposition~\ref{prop:fibered}, the knot group $G(K^5)$ is of type $F_\infty$.
This can also be seen directly as follows.
In \cite{Zee65}, it is shown that $G(K^5)$ is isomorphic to $2I \times \Z$, where $2I$ denotes the binary icosahedral group.
Since $2I\times \Z$ contains $\Z$ as a finite-index subgroup, it is of type $F_\infty$ by \cite[Corollary~7.2.4]{Geo08}.
\end{example}

%%%
\subsection{$2$-plat $2$-knots}
In this subsection, we prove Theorem~\ref{thm:2-plat} and discuss the $2$-knot groups for several examples.
We first recall the definition and notation from \cite{Yas26}.
A $2$-knot $F$ in $S^4$ is called \emph{$n$-plat} if it is equivalent to the plat closure of some $2$-dimensional $2n$-braid, where the plat closure is defined by capping off the $2n$ circles with $n$ annuli.
See \cite[Section~2.1]{Yas26} for the detailed construction.

Let $p$ and $q$ be coprime integers, where $p > 0$ is odd.
Consider a continued fraction expansion 
\[
\frac{q}{p} = \frac{1}{c_1 + \frac{1}{c_2 + \cdots +\frac{1}{c_m}}},
\]
where $m$ is odd.
Let $F(p, q)$ denote the oriented ribbon $2$-knot defined by a ribbon presentation illustrated in \cite[Figure~2]{Yas26}.
Yasuda proves in \cite[Theorem~1.1]{Yas26} that every oriented $2$-plat $2$-knot is equivalent to some $F(p, q)$, and vice versa.

In \cite[Lemma~3.11]{Yas26}, he shows that $F(p, q)$ and $F(p, q')$ are equivalent if $q \equiv q' \bmod p$. 
Hence, we may assume that $q$ is even.
Under this assumption, he also proves \cite[Theorem~4.4]{Yas26} that $F(p, q)$ has a ribbon type $R(\varepsilon_1, \dots, \varepsilon_{p-1})$, where $\varepsilon_i = (-1)^{\lfloor iq / p \rfloor}$.
This information allows us to use \cite[Lemma~3]{Mar77} which gives a presentation of the fundamental group of the exterior $E(F(p,q))$:
\[
G(F(p,q)) \cong \ang{x, y \mid wxw^{-1}y^{-1} }, \quad
w=x^{\varepsilon_1} y^{\varepsilon_2} \cdots x^{\varepsilon_{p-2}} y^{\varepsilon_{p-1}}.
\]
Then, the homomorphism $\ab \colon G(F(p,q)) \to \Z$ defined by $\ab(x) = \ab(y) =1$ gives the abelianization, and $S(G(F(p, q))) = \{[\ab], [-\ab]\}$.

\begin{lemma}
\label{lem:parity}
Let $p$ and $q$ be coprime integers with $p > 0$ odd and $q$ even.
Then, $\varepsilon_i + \varepsilon_{p-i} =0$ for $i=1, \dots, p-1$.
\end{lemma}

\begin{proof}
The above equality is equivalent to the assertion that $\lfloor iq / p \rfloor + \lfloor (p-i)q / p \rfloor$ is odd.
Since $(p-i)q / p$ is not an integer, $\lfloor (p-i)q / p \rfloor = q + \lfloor -iq / p \rfloor = q -\lfloor iq / p \rfloor -1$.
This completes the proof since $q$ is even.
\end{proof}

Lemma~\ref{lem:parity} implies that the sequence $\ab(1), \ab(x^{\varepsilon_1}), \dots, \ab(x^{\varepsilon_1} y^{\varepsilon_2} \cdots x^{\varepsilon_{p-2}} y^{\varepsilon_{p-1}})$ is symmetric (palindromic).
For example, if $(p, q) = (7, 10)$, then $\varepsilon_1, \dots, \varepsilon_6$ are $-1, 1 ,1, -1, -1, 1$ in this order, and thus the above sequence is $0, -1, 0, 1, 0, -1, 0$.

\begin{proof}[Proof of Theorem~\ref{thm:2-plat}]
By \cite[Lemma~3.11]{Yas26} mentioned above, it suffices to consider $F(p, q)$ for $p>0$ odd and $q$ even.
In addition, we assume $p \geq 3$ since $F(1, q)$ is a trivial $2$-knot.
Let $G = G(F(p, q))$.
For convenience, we write $s_1 \cdots s_{2p}$ for the word $x^{\varepsilon_1} \cdots y^{\varepsilon_{p-1}} x y^{-\varepsilon_{p-1}} \cdots x^{-\varepsilon_1} y^{-1}$.
To prove $\Sigma^1(G) \subsetneq S(G)$, we use Theorem~\ref{thm:Bro87} for the sequence 
\begin{align}
\chi(s_1), \chi(s_1 s_2), \dots, \chi(s_1 \cdots s_{2p}), 
\label{eq:sequence}
\end{align}
where $\chi = \pm \ab$.
Let $p = 2k+1$ and define $h(i) \in \Z$ by $h(0)=0$ and $h(i) = \varepsilon_1+\cdots+\varepsilon_i$ for $1\leq i\leq p-1$.
We consider two cases (I) $h(i)$ does not attain the minimum at $i=k$ or (II) $h(i)$ does not attain the maximum at $i=k$.
Note that (I) or (II) must occur since $\{h(i) \mid 0\leq i\leq p-1\}$ contains more than one element.

In the case (I), $h(i)$ attains the minimum at least twice, say $i = j, p-j-1$ $(0\leq j< k)$.
In this case, we consider $\chi = \ab$ and the sequence \eqref{eq:sequence} is
\begin{align*}
h(1), \dots, h(p-2), 0, 1, h(1)+1, \dots, h(p-2)+1, 1, 0.
\end{align*}
See Example~\ref{ex:minimum} below for example.
We can check that this sequence attains its minimum at the $j$th and $(p-j-1)$st terms if $j\geq 1$, at the $(p-1)$st and $2p$th terms if $j=0$.
Hence, $[\ab] \notin \Sigma^1(G(F(p, q)))$.

In the case (II), $h(i)$ attains the maximum at least twice, say $i = j, p-j-1$.
By an argument similar to that in case (I), the sequence \eqref{eq:sequence} for $\chi = -\ab$ attains its minimum at the $(p+j)$th and $(2p-j-1)$st terms if $j\geq 1$, at the $p$th and $(2p-1)$st terms if $j=0$.
Hence, $[-\ab] \notin \Sigma^1(G(F(p, q)))$.
\end{proof}

\begin{example}
\label{ex:minimum}
Let $(p,q) = (7,10)$.
Then, the sequence \eqref{eq:sequence} for $\chi =\ab$ is 
\[
-1, 0, 1, 0, -1, 0, 1, 0, 1, 2, 1, 0, 1, 0,
\]
which is illustrated in Figure~\ref{fig:character-values-7-10}.
It attains the minimum at $i=1$ and $i=7-1-1=5$.
\end{example}

\begin{figure}[h]
\centering
\begin{tikzpicture}[
  x=0.65cm, y=0.85cm,
  >=stealth,
  line cap=round, line join=round,
  font=\small
]
  % Axes: H_i = ab(s_1 ... s_i).
  \draw[->, semithick] (-0.15,0) -- (15.1,0)
    node[right] {$i$};
  \draw[->, semithick] (0,-1.5) -- (0,2.55);
%    node[above] {$H_i$};

  % Horizontal-axis labels.
  \draw (1,-0.06) -- (1,0.06);
  \node[above] at (1,0.10) {$1$};

  \foreach \x in {7,14} {
    \draw (\x,-0.06) -- (\x,0.06);
    \node[below] at (\x,-0.10) {$\x$};
  }

  % Vertical-axis labels.
  \foreach \y in {-1,1,2} {
    \draw (-0.08,\y) -- (0.08,\y);
    \node[left] at (-0.10,\y) {$\y$};
  }
  \node[below left] at (0,0) {$0$};

  % Red: the part coming from w.
  \draw[red, thick]
    (1,-1) -- (2,0) -- (3,1)
           -- (4,0) -- (5,-1) -- (6,0);

  % Black: the letter x between w and w^{-1}.
  \draw[black, thick]
    (6,0) -- (7,1);

  % Red: the part coming from w^{-1}.
  \draw[red, thick]
    (7,1) -- (8,0) -- (9,1) -- (10,2)
          -- (11,1) -- (12,0) -- (13,1);

  % Black: the final letter y^{-1}.
  \draw[black, thick]
    (13,1) -- (14,0);

  % Optional: include the initial segment from H_0=0 to H_1=-1.
  % \draw[red, thick] (0,0) -- (1,-1);

\end{tikzpicture}
\caption{The values $\ab(s_1\cdots s_i)$ for $(p, q)=(7, 10)$. The red segments correspond to $w$ and $w^{-1}$ in $wxw^{-1}y^{-1}$.}
\label{fig:character-values-7-10}
\end{figure}

We now discuss particular examples of $G(F(p, q))$.

\begin{example}
\label{ex:BS}
Let $(p, q) = (2m+1, 2m)$ for $m\geq 1$.
Then, $\lfloor \frac{2mi}{2m+1} \rfloor = i +\lfloor \frac{-i}{2m+1} \rfloor$, and thus $\varepsilon_i = (-1)^{i-1}$ and $w = (xy^{-1})^m$.
By Theorem~\ref{thm:Bro87}, we have
\[
\Sigma^1(G(F(p, q))) = 
\begin{cases}
 \{[-\ab]\} & \text{if $m=1$,} \\
 \emptyset & \text{if $m\geq 2$.}
\end{cases}
\]
Furthermore, putting $t = x$ and $a = yx^{-1}$, we have
\[
w x w^{-1} y^{-1} = a^{-m} t a^m (at)^{-1} = a^{-m} t a^{m} t^{-1} a^{-1}.
\]
Therefore, $G(F(p, q))$ is isomorphic to the Baumslag--Solitar group 
\[
\BS(m, m+1) = \ang{a, t \mid t a^m t^{-1} a^{-m-1}}.
\]
The above result on $\Sigma^1(G(F(p, q)))$ coincides with a known result on $\BS(m, m+1)$.
See, for example, \cite[Sections~A2.1a and B3.3d]{Str13} and \cite[Section~18.3.B]{Geo08}.
\end{example}

%%%%
\section{Further differences between classical knot groups and $2$-knot groups}
\label{subsec:difference}

As seen in the previous sections, knot groups of $2$-knots and those of classical knots can behave differently with respect to the BNS invariant.
We finally observe another group-theoretic difference between them.

Let $\defy(G)$ denote the deficiency of a group $G$, that is, the maximum of $s-r$, where $G$ admits a presentation with $s$ generators and $r$ relators.
Let $\cd(G)$ denote the cohomological dimension of $G$.

\begin{lemma}[{\cite[Lemma~1]{Hil13}}]
\label{lem:Hil13}
Let $G$ be a finitely presentable group such that $[G,G]$ is finitely generated and $G^\ab\cong \Z$.
Then, the following are equivalent.
\begin{enumerate}[label=\textup{(\roman*)}]
\item $[G,G]$ is a free group.
\item $\cd(G)\leq 2$ and $[G,G]$ is of type $\FP_2$ over $\Z$.
\item $\defy(G)=1$.
\end{enumerate}
\end{lemma}

\begin{proposition}
Let $K$ be a fibered $2$-knot such that the fundamental group of its fiber is not a free group.
Then, the geometric dimension of $G(K)$ is greater than or equal to $3$ and $\defy(G(K))\leq 0$.
For instance, a $k$-twist-spun knot $K_0^k$ satisfies the condition if $K_0$ is a prime knot in $S^3$ and $|k| \geq 2$.
\end{proposition}

\begin{proof}
Let $G = G(K)$.
Then we have that $[G, G]$ is of type $\FP_2$ and $\defy(G)\leq 1$.
Since $[G, G]$ is not a free group, the assertion is a direct corollary of Lemma~\ref{lem:Hil13}.

Let us observe that $K = K_0^k$ satisfies the condition.
Suppose, to the contrary, that $[G, G]$ is a free group.
Then, by \cite[Theorem~5.2]{Hem04}, the orientable $3$-manifold $\Sigma_{|k|}(K_0)$ is homeomorphic to $S^3$ or a connected sum of several $S^1\times S^2$.
On the other hand, since $K_0$ is prime, the equivariant sphere theorem shows that $\Sigma_{|k|}(K_0)$ is irreducible (see \cite{Dun85}).
Moreover, it is not homeomorphic to $S^3$ by the proof of the Smith conjecture.
These lead to a contradiction.
\end{proof}

%%%
\subsection{Links in $S^3$}
\label{subsec:link}

In this subsection, we summarize some results on the BNSR invariants of links in $S^3$ and prove Proposition~\ref{prop:Hopf}.
For a compact, connected, orientable $3$-manifold $Y$, let $\Sigma(Y)$ denote the subset of $S(\pi_1(Y))$ consisting of the classes $[\chi]$ represented by closed $1$-forms $\omega$ such that both $\omega$ and its restriction to $\partial Y$ are nowhere-vanishing.
We recall from \cite{BNS87} a result which connects $\Sigma(Y)$ and the BNS invariant.

\begin{theorem}[{\cite[Theorem~E]{BNS87}}]
\label{thm:BNS_E}
Let $Y$ be a compact, connected, oriented $3$-manifold whose boundary is empty or a union of tori.
Then, $\Sigma^1(\pi_1(Y)) = \Sigma(Y)$.
\end{theorem}

In light of the resolution of the Poincar\'{e} conjecture, the statement has been simplified.

\begin{remark}
The boundary condition is crucial.
If $Y$ is a $3$-torus with a $3$-ball removed, then Theorem~\ref{thm:BNS_E} does not hold.
Indeed, $\Sigma^1(\pi_1(Y)) = S(\Z^3) = S^2$, while $\Sigma(Y) = \emptyset$.
\end{remark}

Furthermore, the BNSR invariants of $\pi_1(Y)$ are given as follows.

\begin{theorem}[{\cite[VII.2, Satz~2.3]{Ren88}}]
\label{thm:Ren88}
Let $Y$ be a compact, connected, oriented, irreducible $3$-manifold whose boundary is empty or a union of tori.
Then, $\Sigma^m(\pi_1(Y)) = \Sigma^m(\pi_1(Y); R) = \Sigma(Y)$ for $m\geq 2$ and a non-zero commutative ring $R$ with unity.
\end{theorem}

Therefore, for a non-split link $L$ in $S^3$, the computation of $\Sigma^m(G(L))$ is reduced to that of $\Sigma^1(G(L))$ or $\Sigma(E(L))$.
Note that if $L$ is a split link, then $G(L)$ is a free product of two non-trivial groups, and $\Sigma^1(G(L)) = \emptyset$ by \cite[Section~A2.1a, 3)]{Str13} for instance.
After preparing a lemma below, we prove Proposition~\ref{prop:Hopf} which answers the question about when $\Sigma^1(G(L))$ coincides with $S(G(L))$.

\begin{lemma}
\label{lem:fg}
Let $N$ be a normal subgroup of a group $G$ containing $[G,G]$.
If $G$ and $[G,G]$ are finitely generated, then $N$ is finitely generated.
\end{lemma}

\begin{proof}
Since $[G,G]$ and $N/[G,G]$ are finitely generated, the exact sequence $1 \to [G,G] \to N \to N/[G,G] \to 1$ implies that $N$ is finitely generated.
\end{proof}

\begin{proof}[Proof of Proposition~\ref{prop:Hopf}]
Let $L$ be a link in $S^3$ with at least two components and let $G = G(L)$.
If $L$ is a Hopf link, then $G\cong \Z^2$, and thus $\Sigma^1(G) = S(G)$.

Suppose $\Sigma^1(G) = S(G)$.
In particular, $L$ is non-split.
Theorem~\ref{thm:homotopical_BNS} implies that $[G,G]$ is finitely generated.
Let $N$ be the kernel of the homomorphism $G\to \Z$ sending meridians to $1$.
Lemma~\ref{lem:fg} implies that $N$ is also finitely generated.
Then, Stallings' fibering theorem~\cite{Sta62} shows that $L$ is a fibered link and $N$ is a free group.

Suppose, to the contrary, that the rank of the free group $N$ is at least $2$. 
Then, $[G,G]$ is a non-trivial finitely generated normal subgroup of the free group $N$.
Hence, it must be a subgroup of finite index in $N$.
On the other hand, we have an exact sequence
\[
1 \to N/[G,G] \to G/[G,G] \to G/N \to 1,
\]
and thus $N/[G,G]$ is a free abelian group of rank $l-1$, where $l$ is the number of components of $L$.
This is a contradiction.

Therefore, $N$ is the infinite cyclic group.
This implies that a fiber surface of $L$ is an annulus, and the monodromy must be the $\pm 1$ twist along the core since $H_1(S^3; \Z) = 0$.
We conclude that $L$ is a Hopf link.
\end{proof}

\appendix

%%%%
\section{The Thurston norm and fibered faces}

In this appendix, we give another proof of Proposition~\ref{prop:Hopf} from the viewpoint of the Thurston norm.
We refer the reader to \cite{Thu86} and \cite{Kit22}.

Throughout this appendix, let $Y$ be a compact, connected, oriented $3$-manifold and let $G = \pi_1(Y)$. 
For $\phi \in H^1(Y; \Z)$, define the Thurston norm $x_Y(\phi) \in \Z$ by
\[
x_Y(\phi) = \min\{ \chi_{-}(S) \mid \text{$S$ is a properly embedded oriented surface in $Y$ dual to $\phi$} \}.
\]
Here, $\chi_{-}(S) = \sum_i \max\{-\chi(S_i), 0\}$, where the sum runs over all connected components of $S$.
It is known that $x_Y(\phi)$ extends to a seminorm on $H^1(Y; \R)$.
The set $B_Y = \{\phi \in H^1(Y; \R) \mid x_Y(\phi)\leq 1\}$ is called the \emph{Thurston norm ball}, which is known to be a (possibly non-compact) convex polyhedron in $H^1(Y; \R)$.

Then, Thurston~\cite[Theorem~5]{Thu86} proved the following theorem.

\begin{theorem}
\label{thm:Thu86}
Let $Y$ be a $3$-manifold fibering over a circle whose fiber has negative Euler characteristic.
Then, there exist top-dimensional faces $F_i$ of $B_Y$ such that $\phi \in H^1(Y; \Z)$ is fibered if and only if $\phi \in \bigcup_i \Int \Cone(F_i)$.
Furthermore, $\Sigma(Y) = \bigcup_i \Int \Cone(F_i) /\R_{> 0}$.
\end{theorem}

Here, $\Int \Cone(F_i)$ denotes the interior of the cone on the face $F_i$.
See also \cite[Theorem~5.2.8]{Kit22}.
Combining Theorems~\ref{thm:Thu86} and \ref{thm:BNS_E}, one obtains the next result.

\begin{corollary}
\label{cor:Thu86}
Let $Y$ be a $3$-manifold fibering over a circle whose fiber has negative Euler characteristic.
If $\dim H^1(Y; \R)\geq 2$, then $\Sigma^1(\pi_1(Y)) \subsetneq S(\pi_1(Y))$.
\end{corollary}

\begin{proof}
It suffices to show that $\bigcup_i \Int \Cone(F_i) \subsetneq H^1(Y; \R)\setminus\{0\}$, which is implicitly established in the proof of \cite[Corollary in Section~3]{Thu86}.
If two distinct faces of $B_Y$ intersect, then the intersection is not contained in the above union.
If no two distinct faces intersect, then $B_Y$ must be bounded by exactly two parallel facets, and thus the union of the corresponding open cones is strictly smaller than $H^1(Y; \R)\setminus\{0\}$.
\end{proof}

\begin{proof}[Alternative proof of Proposition~\ref{prop:Hopf}]
The Hopf link case is immediate.
Conversely, suppose that $\Sigma^1(G(L)) = S(G(L))$.
By the same argument as in the first part of the proof in Section~\ref{subsec:link}, $L$ is fibered.
Let $F$ be its fiber surface.
If $\chi(F)<0$, then Corollary~\ref{cor:Thu86} leads to a contradiction.
Hence, $\chi(F)\geq 0$, that is, $F$ is an annulus and $L$ must be a Hopf link.
\end{proof}

\subsection*{Use of artificial intelligence}
The author used ChatGPT (GPT-6 Astra), provided free of charge through OpenAI's ChatGPT for Academic Researchers program, to assist with proofreading, reviewing mathematical arguments, and preparing TikZ code. 
The author reviewed the output and takes full responsibility for the manuscript.

\end{document}